\documentclass[a4paper,10pt]{article}
\usepackage[utf8]{inputenc}

\usepackage{comment,xcomment,amssymb,amsmath,color}
\usepackage{xypic,array}
\usepackage{latexsym,mathtools,amsthm}
\usepackage{enumerate}
\usepackage{multirow}
\usepackage{booktabs}
\title{Einstein nilpotent Lie groups}
\author{Diego Conti and Federico A. Rossi}

\newtheorem{theorem}{Theorem}[section]
\newtheorem{lemma}[theorem]{Lemma}

\newtheorem{remark}[theorem]{Remark}
\newtheorem{corollary}[theorem]{Corollary}
\newtheorem{proposition}[theorem]{Proposition}
\newtheorem{example}[theorem]{Example}

\newcommand{\R}{\mathbb{R}}
\newcommand{\lie}[1]{\mathfrak{#1}}     
\newcommand{\g}{\lie{g}}

\newcommand{\C}{\mathbb{C}}

 \newcommand{\PP}{\mathcal{P}}

\newcommand{\hook}{\lrcorner\,}

\newcommand{\SO}{\mathrm{SO}}

\newcommand{\GL}{\mathrm{GL}}
\newcommand{\SL}{\mathrm{SL}}

\newcommand{\id}{\mathrm{Id}}   

\newcommand{\gl}{\lie{gl}}
\newcommand{\Sl}{\lie{sl}}
\newcommand{\Span}[1]{\operatorname{Span}\left\{#1\right\}}
\newcommand{\tran}[1]{\hspace{.2mm}\prescript{t\hspace{-.5mm}}{}{#1}}
\DeclareMathOperator{\ric}{ric}
\DeclareMathOperator{\Der}{Der}
\DeclareMathOperator{\Ric}{Ric}

\DeclareMathOperator{\End}{End}
\DeclareMathOperator{\Aut}{Aut}

\DeclareMathOperator{\ad}{ad}

\DeclareMathOperator{\Tr}{tr}

\newcolumntype{C}{>{$}c<{$}}
\newcolumntype{L}{>{$}l<{$}}
\newcolumntype{R}{>{$}r<{$}}

\newcommand{\llangle}{\langle\!\langle}
\newcommand{\rrangle}{\rangle\!\rangle}
\begin{document}
\maketitle
\begin{abstract}
We study the Ricci tensor of left-invariant pseudoriemannian metrics on Lie groups. For an appropriate class of Lie groups that contains nilpotent Lie groups, we introduce a variety with a natural $\GL(n,\R)$ action,  whose orbits parametrize Lie groups with a left-invariant metric; we show that the Ricci operator can be identified with the moment map relative to a natural symplectic structure. From this description we deduce that the Ricci operator is the derivative of the scalar curvature $s$ under gauge transformations of the metric, and show that Lie algebra derivations with nonzero trace obstruct the existence of Einstein metrics with $s\neq0$.

Using the notion of nice Lie algebra, we give the first example of a left-invariant Einstein metric with $s\neq0$ on a nilpotent Lie group. We show that nilpotent Lie groups of dimension $\leq 6$ do not admit such a metric, and a similar result holds in dimension $7$ with the extra assumption that the Lie algebra is nice.
\end{abstract}

\renewcommand{\thefootnote}{\fnsymbol{footnote}} 
\footnotetext{This work was partially supported by GNSAGA of INdAM.}
\renewcommand{\thefootnote}{\arabic{footnote}}

The construction of homogeneous Einstein metrics is a classical problem in Riemannian geometry, which is mostly relevant in the case of nonzero scalar curvature, since Ricci-flat homogeneous metrics are necessarily flat \cite{AlekseevskiKimelFel}.

The case of positive scalar curvature corresponds to compact homogeneous spaces $G/K$, due to the well-known theorems of Bochner \cite{Bochner:vector_fields} and Myers \cite{Myers}. In this setting, Einstein metrics arise as critical points of the total scalar curvature functional; this approach has been pursued constructively in \cite{WangZiller}. Whilst a classification of compact homogeneous Einstein manifolds has not yet been obtained, both  necessary and sufficient conditions on a compact $G/K$ for the existence of an Einstein Riemannian metric are known (see \cite{BohmWangZiller, Bohm:nonexistence}).

All known examples of negative scalar curvature arise as solvmanifolds, i.e. left-invariant metrics on solvable Lie groups; in fact, a conjecture of Alexseevsky states that all Riemannian homogeneous Einstein manifolds of negative scalar curvature are of this type. The algebraic structure of these Lie groups is well understood (see \cite{Heber:noncompact,Lauret:Einstein_solvmanifolds});  they are non-compact in an essential way, in the sense that  they are necessarily not unimodular (as originally proved in \cite{Dotti}), and therefore do not admit a cocompact lattice.

This tripartition does not generalize to the pseudoriemannian case. For indefinite metrics, the conditions of positive and negative scalar curvature are somewhat interchangeable, being that a metric $g$ and its opposite $-g$ have the same Ricci tensor. Moreover, there is no direct relation between compactness and scalar curvature: both compact and non-compact Lie groups with Einstein indefinite metrics  with either zero or nonzero scalar curvature are known  \cite{CalvarusoZaeim:neutral, CalvarusoZaeim:lorentzian,Derdzinski}. Other non-compact homogeneous examples appear in the context of parak\"ahler and nearly parak\"ahler geometry (see \cite{AlekseevskyMedoriTomassini:Homogeneous2009,IvanovZamkovoy:parahermitian}).

This greater flexibility is evident in the context of Einstein solvmanifolds, which in the indefinite case are allowed to  have either zero or nonzero scalar curvature, and can admit cocompact lattices. A source of Ricci-flat manifolds is given by bi-invariant metrics on nilpotent Lie groups (see \cite{globke}); if the structure constants are rational, this determines a homogeneous metric on a compact quotient. By relaxing the condition and considering left-invariant metrics, it is possible to obtain  more examples of compact Ricci-flat nilmanifolds (see \cite{Freibert:calibrated, FinoLujan, ContiRossi:ricci}. Some of these examples are not flat, indicating another difference with the Riemannian case. The case of pseudoriemannian Einstein solvmanifolds with nonzero scalar curvature appears to be unexplored at the time of writing, except for the four-dimensional study of \cite{CalvarusoZaeim:lorentzian,CalvarusoZaeim:neutral}.

\smallskip

In this paper we study Einstein pseudoriemannian metrics of nonzero scalar curvature on nilpotent Lie groups; we only consider left-invariant metrics, and  in the rest of this introduction all metrics are implicitly left-invariant. The starting observation is that there are two relevant group of symmetries, namely $\SO(p,q)$ and $\GL(n,\R)$. In the first section, we consider general left-invariant metrics on a Lie group; we exploit the $\SO(p,q)$-invariance to express the Ricci curvature in terms of appropriate equivariant maps. Necessary conditions on the Lie algebra of a nilpotent Lie group for the existence of  an Einstein metric of nonzero scalar curvature are given in Section~\ref{sec:fivedim}; in particular, we show that it must have step $\geq 3$ and that the centre must be contained in the derived algebra (Corollary~\ref{cor:steptwo} and Lemma~\ref{lemma:centroinderivato}). More generally, we focus on the class of unimodular Lie groups with Killing form zero; this condition is intermediate between nilpotent and solvable, and gives the Ricci curvature a particularly simple expression (Proposition~\ref{prop:ricci_in_peace}). 

The set of Lie algebra structures on $\R^n$ is a subvariety of $\Lambda^2(\R^n)^*\otimes \R^n$ (in particular, the subvariety of nilpotent Lie algebras has been studied in \cite{Goze:OnTheVarieties,Lauret}); the unimodular condition singles out an irreducible $\GL(n,\R)$-submodule $V^{11}_1\subset\Lambda^2(\R^n)^*\otimes \R^n$. The choice of a scalar product on $\R^n$ determines an isomorphism $\R^n\cong(\R^n)^*$, and therefore a Lie algebra structure on $(\R^n)^*$. Thus, every $n$-dimensional unimodular Lie group with a left-invariant metric determines a pair  $(a,b)\in V^{11}_1\times (V^{11}_1)^*$, uniquely defined up to $\GL(n,\R)$ action. The Ricci operator, which via $\g\cong\R^n$ is identified with an element of $\gl(n,\R)$, corresponds then to a specific bilinear map
\[\mu\colon V^{11}_1\times (V^{11}_1)^*\to\gl(n,\R)\]
(see Proposition~\ref{prop:c1c2}). We prove that $\mu$ is the moment map for the action of $\GL(n,\R)$, relative to the symplectic structure on $V^{11}_1\times (V^{11}_1)^*$ induced by the natural pairing (Proposition~\ref{prop:moment}). This description, valid both in the Riemannian and the indefinite case, gives a symplectic interpretation of the known identification of the Ricci tensor as a moment map in the sense of real geometric invariant theory (see \cite{Lauret}).

We deduce that on a fixed Lie group, the component $\langle \mu,X\rangle$ obtained by taking the natural pairing with $X\in\gl(n,\R)$, can be identified with the derivative of the scalar curvature as the metric is deformed in the direction of $X$; for $X=\id$, we recover a result of Jensen \cite{Jensen:scalar}, according to which Einstein metrics are critical points for the scalar curvature among metrics of a fixed volume. Another remarkable consequence is that Einstein metrics of nonzero scalar curvature may only exist on a Lie group if all derivations of its Lie algebra are tracefree (Theorem~\ref{thm:ostruzionegl}). This is used to prove that nilpotent Lie groups of dimension $\leq 6$ admit no such metric. The proof is a simple linear computation on a case-by-case basis that uses the classification of \cite{Magnin}.

In the final section we consider nilpotent Lie algebras that admit a nice basis (see \cite{LauretWill:Einstein}); the key property of such a basis is that a metric which is diagonal relative to the basis has a diagonal Ricci tensor. This makes it possible to construct the first example of an Einstein  metric of nonzero scalar curvature on a nilpotent Lie group (Theorem~\ref{thm:einstein_example}). 

Our example is in dimension 8, and this is the lowest dimension in which this approach is fruitful. In fact, using the classification of \cite{Gong} we prove that 7-dimensional nilpotent Lie groups with a nice basis admit a derivation with nonzero trace, and therefore do not carry any Einstein metric of nonzero scalar curvature. We do not know whether any 7-dimensional nilpotent Lie group admits an Einstein metric of nonzero scalar curvature, but we are able to prove that if such an example exists then its Lie algebra must be one of the 11 Lie algebras listed in Theorem~\ref{thm:einstein_seven_in_table}.

\section{Left-invariant pseudoriemannian metrics}
In this section we reproduce the known formula for the Ricci tensor of a left-invariant metric on a Lie group (see e.g. \cite{Milnor:curvatures}), which we revisit from the viewpoint of representation theory by making use of the natural $\SO(p,q)$-invariance. The formulae established in this section will form the basis for subsequent calculations.

Let $G$ be a Lie group with Lie algebra $\g$; a left-invariant pseudoriemannian metric on $G$ can be identified with a non-degenerate bilinear form $\llangle\,,\rrangle$ on $\g$; the Levi-Civita connection and its curvature $R$ are then elements of the tensor algebra over $\g$. We shall refer to $\llangle,\rrangle$ as a (pseudoriemannian) metric on $\g$, and to $R$ as its curvature.

In this section, such a metric will be identified by an orthonormal basis $\{e_i\}$ of $\g$, where $\llangle e_i,e_i\rrangle=\epsilon_i=\pm1$. Writing $\llangle [e_i,e_j],e_k\rrangle =c_{ijk}$, the Koszul formula gives
\[\nabla_{e_i} e_j=\frac12( c_{ijk}+c_{kij}+c_{kji})\epsilon_k e_k;\]
here and throughout the paper, summation over repeated indices is implied.

The Riemann tensor is therefore given by
\begin{multline*}
R(e_i,e_j,e_h,e_i)
=\frac12 \epsilon_k  c_{iki}( c_{jhk}+c_{kjh}+c_{khj})\\
-\frac14\epsilon_k (c_{ihk}+c_{khi}+c_{kih})(c_{jki}+c_{ikj}+c_{ijk})\\
-\frac12\epsilon_k c_{ijk}(c_{khi}+c_{ihk}+c_{ikh}),
\end{multline*}
and the Ricci tensor by
\begin{multline}
\label{eqn:ricci_indices}
\Ric(e_j,e_h)=\sum_i \epsilon_i R(e_i,e_j,e_h,e_i)
=\frac12 \epsilon_k \epsilon_i c_{iki}(c_{kjh}+c_{khj})+\frac14\epsilon_k \epsilon_ic_{ikh}c_{ikj}\\
-\frac12\epsilon_k\epsilon_i c_{ijk}c_{khi}+\frac12 \epsilon_k \epsilon_i c_{iki}c_{jhk}-\frac12\epsilon_k\epsilon_i c_{ijk}c_{ihk}.
\end{multline}
Note that the term $\epsilon_k \epsilon_i c_{iki}c_{jhk}$ vanishes because of the well-known identity
\begin{equation}
\label{eqn:Killingsymmetric}
\Tr \ad [e_j,e_k]=0.
\end{equation}
Denote by $U_{a_1,\dotsc, a_k}$ the irreducible real representation of $\SO(p,q)$ with maximal weight $a_1L_1+\dots + a_kL_k$; let $T=U_1$ be the standard representation of $\SO(p,q)$. Having fixed a metric on $\g$, an isomorphism $\g\cong T$ is induced; the Lie bracket gives then an element of $\Lambda^2T^*\otimes T$. More precisely, the Lie bracket satisfies the Jacobi identity, so it lies in the zero set of a quadratic map $\Lambda^2T^*\otimes T\to \Lambda^3T^*\otimes T$, which we view as a linear equivariant map
\begin{equation}
\label{eqn:jacobi}
S^2(\Lambda^2T^*\otimes T)\to \Lambda^3T^*\otimes T.
\end{equation}
In these terms, the identity \eqref{eqn:Killingsymmetric} can be explained by the fact that $\Lambda^3T^*\otimes T$ contains a copy of $\Lambda^2T^*$.

Thus, the structure constants $\{c_{ijk}\}$ belong to the $\SO(p,q)$-module $\Lambda^2T^*\otimes T\cong \Lambda^3T^*+T^*+U_{21}$, where the first two components correspond to
\[\frac12(c_{ijk}-c_{ikj})e^{ijk}, \quad \epsilon_j c_{ijj} e^i.\]
Much like the Jacobi identity, the Ricci  tensor as a function of the $\{c_{ijk}\}$ can be viewed as a linear equivariant map
\[S^2(\Lambda^2T^*\otimes T)\to S^2T^*=\R+U_2.\]
Notice that the linear constraints on the domain of this map established by the Jacobi identity are transparent to this map, as $\Lambda^3T^*\otimes T$ does not contain submodules isomorphic to either $\R$ or $U_2$.

The $\SO(p,q)$-module $S^2(\Lambda^2T^*\otimes T)$ contains three copies of $\R$ (one for each irreducible submodule of $\Lambda^2T^*\otimes T)$ and six copies of $U_2$, namely
\begin{gather*}
U_2\subset T\otimes U_{21}, \quad U_2\subset S^2T, \quad 2U_2\subset S^2(U_{21}),\\
U_2\subset S^2(\Lambda^3T^*),\quad U_2\subset U_{21}\otimes\Lambda^3T^*.
\end{gather*}
Therefore, to a Lie algebra with a fixed metric one can associate three scalar invariants and six trace-free bilinear forms. For an explicit representation, it is more natural to define six bilinear forms:
\begin{align*}
B_1(v,w)&=\Tr \ad (v\hook dw^\flat+w\hook dv^\flat)^\sharp,&
B_2(v,w)&= \Tr(\ad(v))\Tr(\ad(w)), \\ 
B_3(v,w)&=\llangle \ad(v),\ad(w)\rrangle,&B_4(v,w)&= \Tr(\ad(v)\circ \ad(w)), \\
B_5(v,w)&= \llangle dv^\flat, dw^\flat\rrangle,& B_6(v,w)&=\llangle \ad(v)^\flat, dw^\flat\rrangle+\llangle \ad(w)^\flat, dv^\flat\rrangle.
\end{align*}
The scalar invariants can then be recovered by taking the trace of $B_2$, $B_3$ and $B_4$.

\begin{lemma}
\label{lemma:ricci_in_general}
The Ricci tensor of a left-invariant pseudoriemannian metric on a Lie group is given by
\begin{multline*}
\Ric(v,w)=-\frac12 \Tr \ad (v\hook dw^\flat+w\hook dv^\flat)^\sharp+\frac12\llangle dv^\flat, dw^\flat\rrangle\\
-\frac12\llangle \ad(v),\ad(w)\rrangle-\frac12\Tr(\ad(v)\circ \ad(w)).
\end{multline*}
\end{lemma}
\begin{proof}
The bilinear forms can be expressed as
\begin{align*}
B_1(e_j,e_h)&=c_{iki}(c_{jkh}+c_{hkj})\epsilon_i\epsilon_k &
B_2(e_j,e_h)&=c_{jii}c_{hkk}\epsilon_i\epsilon_k\\
B_3(e_j,e_h)&=c_{jik} c_{hik}\epsilon_i\epsilon_k&
B_4(e_j,e_h)&=c_{hik} c_{jki} \epsilon_i \epsilon_k \\
B_5(e_j,e_h)&=\frac12c_{ikj}c_{ikh} \epsilon_i \epsilon_k&
B_6(e_j,e_h)&=-c_{jik}c_{ikh}\epsilon_i\epsilon_k.
\end{align*}
By \eqref{eqn:ricci_indices}, we obtain
\[\Ric(v,w)=-\frac12B_1+\frac12B_5-\frac12B_3-\frac12B_4.\qedhere\]
\end{proof}

\begin{remark}
In the case of a bi-invariant pseudoriemannian metric, one has 
\[\llangle [v,x], y\rrangle +\llangle x,[v,y]\rrangle=0,\]
meaning that $\{c_{ijk}\}$ belongs to $\Lambda^3T^*$. The invariants then satisfy
\[B_1=B_2=0, \quad B_4=-B_3=-2B_5;\]
consistently with \cite{oneill:semiriemannian} one obtains
\[\Ric(v,w)=-\frac14\Tr(\ad(v)\circ \ad(w)),\]
i.e. the Ricci tensor is a multiple of the Killing form. Thus, non-Ricci-flat Einstein bi-invariant metrics only exist on semisimple Lie groups (see also \cite{globke}). On the other hand, nilpotent Lie groups with bi-invariant metrics are necessarily Ricci-flat. Moreover, the existence of a non-degenerate bi-invariant metric puts constraints on the structure constants; for instance \cite{Medina:groupes} shows that non-abelian nilpotent Lie groups do not admit Lorentzian bi-invariant metrics.
\end{remark}

\begin{remark}
Denoting by $Z$ the vector in $\g$ defined by
\[\llangle Z,v\rrangle = \Tr\ad(v),\]
i.e. $Z = \epsilon_i\epsilon_k c_{ikk}e^i$, we can rewrite 
\begin{align*}
-\frac12 \Tr \ad (v\hook dw^\flat+w\hook dv^\flat)^\sharp&=-\frac12 (dw^\flat(v,Z)+dv^\flat(w,Z))\\
&=\frac12(\llangle [v,Z],w\rrangle +\llangle[w,Z],v\rrangle).
\end{align*}
This leads to the formula of \cite{Besse:Einstein1987},
\[\Ric(v,v)=-\frac12\llangle [v,e_i],[v,e_i]\rrangle -\frac12 B(v,v)+\frac14 \llangle [e_i,e_j],v\rrangle ^2 - \llangle [Z,v],v\rrangle.\]
\end{remark}

\section{Nonexistence of Einstein metrics with $s\neq0$}
\label{sec:fivedim}
In this section we  specialize to the nilpotent case, and give examples of nilpotent Lie algebras that do not admit Einstein metrics of nonzero scalar curvature. 

The results of this section are based on the following specialization of Lemma~\ref{lemma:ricci_in_general}:
\begin{proposition}
\label{prop:ricci_in_peace}
For $\g$ a unimodular Lie algebra with Killing form zero, and any metric  $\llangle\,, \rrangle$,
\[\Ric(v,w)=\frac12\llangle dv^\flat, dw^\flat\rrangle -\frac12\llangle \ad(v),\ad(w)\rrangle.\]
\end{proposition}

Let $\g$ be a nilpotent Lie algebra, and let $\g\supset \g^1\supset \dots\supset \g^k \supset \{0\}$ be its lower central series. Then $\g^k$ is contained in the centre $Z$.
\begin{lemma}
\label{lemma:centroinderivato}
Let $\lie{g}$ be a nilpotent Lie algebra with an Einstein metric of nonzero scalar curvature. Then 
\[Z\subset \g^1.\]
\end{lemma}
\begin{proof}
If $v$ is in the centre and $w$ is orthogonal to $\g^1=[\lie{g},\lie{g}]$, then both $\ad v$ and $dw^\flat$ are zero; by Proposition~\ref{prop:ricci_in_peace}, $\Ric(v,w)=0$; because the Ricci tensor is a nonzero multiple of the metric, this implies that $v$ and $w$ are orthogonal. Thus, $Z$ is orthogonal to $(\g^1)^\perp$; by non-degeneracy of the metric, the statement follows.
\end{proof}
It follows from this lemma that Einstein metrics of nonzero scalar curvature cannot exist on reducible Lie algebras of the form $\g\oplus\R^n$, with $\g$ nilpotent.

We will need the following:
\begin{lemma}
\label{lemma:trace}
Let $V$, $W$ be  vector spaces endowed with a scalar product, and assume the scalar product on $V$ is not degenerate. For any linear map $f\colon V\to W$, 
\[\llangle f, f\rrangle = \Tr \tilde h,\]
where $\llangle f,f\rrangle$ is computed relatively to the induced scalar product on $V^*\otimes W$, and 
\[\tilde h\colon V\to V, \quad \tilde h(v)=\llangle f(v), f(\cdot)\rrangle^\sharp.\]
If $U\subset V$ and $\pi_U\colon V\to U$ is any projection (i.e. any left inverse of the inclusion), the trace of the composition
\[U\to V\xrightarrow{\tilde h} V\xrightarrow{\pi_U} U\]
coincides with $\llangle f|_U\circ \pi_U, f\rrangle$.
\end{lemma}
\begin{proof}
Let $e_1,\dotsc, e_n$ be a basis of $V$, and let $(g^{ij})$ be the matrix defining the scalar product on $V^*$. Then
\[\tilde h(u)=g^{ij}\llangle f(u),f(e_j)\rrangle e_i,\]
so
\[\Tr \tilde h=g^{kj}\llangle f(e_k),f(e_j)\rrangle=\llangle e^k\otimes f(e_k), e^j\otimes f(e_j)\rrangle.\]
If $U\subset \R^n$ is spanned by $e_1,\dotsc, e_l$ and $\pi_U$ is the projection that annihilates $e_{l+1},\dotsc, e_n$, then 
\[\pi_u(\tilde h(u))=\sum_{i=1}^l \sum_{j=1}^n g^{ij}\llangle f(u),f(e_j)\rrangle e_i,\]
so the trace of the composition is
\[\sum_{k=1}^{l}\sum_{j=1}^n g^{kj}\llangle f(e_k),f(e_j)\rrangle=\llangle f|_U\circ\pi_U,f\rrangle.\qedhere\]
\end{proof}
This lemma will be applied to the operators
\[\tilde B_3,\tilde B_5\colon \g\to\g, \quad \llangle \tilde B_3(v),w\rrangle = B_3(v,w), \quad \llangle\tilde B_5(v), w\rrangle = B_5(v,w);\]
in particular the first part gives
\begin{equation}
\label{eqn:traceB3B5}
\Tr \tilde B_3=\llangle \ad,\ad\rrangle=2\llangle d,d\rrangle, \quad \Tr \tilde B_5=\llangle d,d\rrangle,
\end{equation}
and therefore
\begin{equation}
\label{eqn:traceofricci}
\Tr \ric = \frac12\llangle d,d\rrangle - \frac12\llangle \ad, \ad\rrangle = -\frac12\llangle d,d\rrangle.
\end{equation}

\begin{proposition}
\label{prop:MandN}
Let $\g$ be a nilpotent Lie algebra and let $g$ be a metric. Let $\ad(\g)$ be the image of $\g$ in $Z^o\otimes \g^1$. Let $d(\g^*)$ be the image of $\g^*$ in $\Lambda^2 Z^o$. Let $M$ be the null space of $\ad(\g)$ and let $N$ be the null space of $d(\g^*)$. If
\[\dim M + \dim N \geq\dim\g^1-\dim Z,\]
then $g$ is not Einstein unless it is Ricci-flat.
\end{proposition}
\begin{proof}
If $g$ is Einstein with nonzero scalar curvature, up to normalization we can assume
\begin{equation}
\label{eqn:einstein_normalized}
 \tilde B_5-\tilde B_3=\id.
\end{equation}
Write $M_\g=\ad^{-1}(M)$, $N_\g=(d^{-1}(N))^\sharp$. We have $\tilde B_5|_{N_\g}=0$, $\tilde B_3|_{M_\g}=0$. Thus, $N_\g$ and $M_\g$ are in direct sum. Since
\[\dim M_\g=\dim Z+\dim M, \quad \dim N_\g=\dim \g-\dim \g^1+\dim N,\]
a dimensional count gives
\[\g=M_\g\oplus N_\g.\]
By \eqref{eqn:einstein_normalized},
\[\tilde B_5|_{M_\g}=\id_{M_\g}, \quad \tilde B_3|_{N_\g}=-\id_{N_\g}.\]
This implies that $\Tr\tilde B_5$ and $\Tr\tilde B_3$ have opposite signs, which contradicts \eqref{eqn:traceB3B5}.
\end{proof}

\begin{corollary}
\label{cor:steptwo}
If $\g$ is nilpotent of step two, all  Einstein metrics are Ricci-flat. 
\end{corollary}
\begin{proof}
The step two condition means that $\g^2=0$, namely $\g^1\subset Z$. If an Einstein metric with nonzero scalar curvature exists, then Lemma~\ref{lemma:centroinderivato} gives $Z\subset \g^1$, and therefore $Z=\g^1$. Proposition~\ref{prop:MandN} gives the statement.
\end{proof}

\begin{remark}
A similar result under slightly stronger assumptions was proved in \cite{DelBarcoOvando}.
\end{remark}

What we have proved so far is sufficient to prove the following:
\begin{proposition}
\label{prop:somenil5}
The nilpotent Lie algebras
\begin{gather*}
(0,0,0,12),\quad (0,0,0,12,14), \quad(0,0,0,0,12), \quad (0,0,12),\quad (0,0,0,0,12+34),\\
(0,0,0,12,13), \quad (0,0,12,13,14), \quad(0,0,12,13)
\end{gather*}
do not admit Einstein metrics of nonzero scalar curvature.
\end{proposition}
\begin{proof}
The statement follows from Lemma~\ref{lemma:centroinderivato} for Lie algebras of the form $\g\oplus\R$, and from Corollary~\ref{cor:steptwo} for those of step two. For the last two Lie algebras in the list, we resort to a different argument. Consider the Lie algebra $(0,0,12,13,14)$. If $e^1$ is light-like, we can use Proposition~\ref{prop:MandN} and conclude. Thus, up to adding the closed form $e^1$ to $e^2,\dots, e^5$ we may assume we have an orthogonal splitting
\[U\oplus C, \quad U=\Span{e_1}, C=\Span{e_2,\dots, e_5}.\]
By Lemma~\ref{lemma:trace},
\[\Tr\tilde B_3|_U = \llangle e^1\otimes \ad e_1, e^1\otimes \ad e_1\rrangle = \llangle d,d\rrangle,\]
where the last equality depends on the structure of the Lie algebra. Thus an Einstein metric necessarily has 
\[\ric=-\llangle d,d \rrangle \id; \]
this contradicts \eqref{eqn:traceofricci}. The same argument applies to $(0,0,12,13)$. 
\end{proof}
This result shows that Einstein metrics are Ricci-flat on all nilpotent Lie algebras of dimension $4$, and on six of the nine nilpotent Lie algebras of dimension $5$, including the abelian case; the remaining cases will be covered in Section~\ref{sec:sixdim}.

\section{The Ricci tensor as a moment map}
\label{sec:momentmap}
In this section we change our point of view; rather than fixing an orthonormal frame, we allow both metric and structure constants to vary simultaneously. Accordingly, the relevant group of symmetries is now $\GL(n,\R)$. Our model for the set of Lie algebras with a fixed metric is a subset in a representation of $\GL(n,\R)$; under suitable assumptions, we will see that the Ricci tensor has a natural symplectic interpretation.

Motivated by Proposition~\ref{prop:ricci_in_peace}, we will consider the class of unimodular Lie algebras on which the Killing form is zero. By Cartan's criterion, this class sits in between nilpotent Lie algebras and solvable Lie algebras. Both inclusions are strict; in particular, we point out that the Killing form of a  unimodular solvable Lie algebra can be nonzero; one example is the Lie algebra $(0,e^{12},-e^{13})$, or $(0,12,-13)$. This notation means that relative to some coframe $\{e^1,e^2,e^3\}$ the exterior derivative reads
\[de^1=0, \quad de^2=e^{12}, \quad de^3=-e^{13},\]
where $e^{12}$ is short-hand notation for $e^1\wedge e^2$.

In this setting, we can restate the Ricci formula as follows. Fix $T=\R^n$ as a $\GL(n,\R)$-module. We have a decomposition
\[\Lambda^2T^*\otimes T \cong T^*\oplus V^{11}_1,\]
where $V^{11}_1$ is an irreducible $\GL(n,\R)$-module with highest weight $L_1-L_{n-1}-L_n$. 

A Lie algebra structure on $T$ is a linear map $\ad\colon T\to\End(T)$; more precisely, it is defined by an element
\[a\in T^*\oplus V^{11}_1=\ker (T\otimes \End T\to S^2T^*\otimes T),\]
which in addition satisfies the Jacobi identity. We shall denote the corresponding Lie algebra by $T_a$; thus, $T_a=T$ as vector spaces. The unimodular condition then reads $a\in V^{11}_1$; our class of Lie algebras therefore corresponds to the variety
\[\mathcal{P}=\{a\in V^{11}_1\mid \text{  $a$ defines a Lie algebra with Killing form zero}\}.\]

The choice of a metric on $T_a$ determines an isomorphism $\flat\colon T\cong T^*$, and thus a Lie algebra structure on $T^*$. This gives rise to a pair
\[(a,b)\in V^{11}_1\times V^1_{11}, \quad V^1_{11}\subset  T\otimes\End T,\]
where $\End(T^*)$ is identified with $\End(T)$ via transposition. Two pairs in the same $\GL(n,\R)$-orbit correspond to the same Lie algebra and metric written with respect to different frames.

Define the contractions
\begin{gather*}
c_1,c_2\colon (T^*\otimes\End T)\otimes (T\otimes\End T)\to \End T,\\
c_1(e^i\otimes a_i, e_j\otimes b_j)=a_i\circ b_i, \quad c_2(e^i\otimes a_i, e_j\otimes b_j)=b_i\circ a_i.
 \end{gather*}
Since elements of $V^{11}_1$, $V^1_{11}$ are skew-symmetric in two indices, $c_2$ can also be written as
\begin{equation}
 \label{eqn:c2new}
\Tr(a_i\circ b_j)e^i\otimes e_j.
\end{equation}
The natural duality $\langle\, ,\rangle \colon V^{11}_1\times V^1_{11}\to\R$ can be expressed as
\begin{equation}
 \label{eqn:tracec1c2}
 \langle a,b\rangle = \Tr c_1(a,b)=\Tr c_2(a,b).
\end{equation}

\begin{proposition}
\label{prop:c1c2}
Given $a\in\PP$, a metric on $T_a$ and $(a,b)$ induced as above, the Ricci operator $\ric\in \gl(T)$ is given by
\[\ric=\frac14c_1(a,b)-\frac12c_2(a,b).\]
\end{proposition}
\begin{proof}
For an endomorphism $u\colon \g\to \g$, define
\[u^\natural\colon \g^*\to \g^*, \quad \sharp \circ u^\natural\circ \flat=u;\]
then
\[\llangle  u_1,u_2\rrangle= \langle u_1, \tran u_2^\natural\rangle = \Tr(u_1\circ \tran u_2^\natural), \quad u_1,u_2\in \End(\g).\]
In particular
\begin{equation}
\label{eqn:musica}
\tran b(e^i)=a((e^i)^\sharp)^\natural.
\end{equation}
Write 
$a(e_h)=a_{hj}^i e^j\otimes e_i$, $b(e^k)=b^{kj}_l e^l\otimes e_j$, so that 
\[de^k=-\frac12 a_{ij}^k e^{ij}, \quad de_h^\flat = -\frac12 b^{jk}_h e_j^\flat \wedge e_k^\flat;\]
Proposition~\ref{prop:ricci_in_peace} then reads
\[2\llangle \ric (e_h),e_k\rrangle =
\llangle de_h^\flat, de_k^\flat\rrangle - \llangle \ad e_h, \ad e_k\rrangle;\]
replacing $e_k$ with $(e^k)^\sharp$,
\begin{multline*}
2e^k(\ric(e_h))=\llangle de_h^\flat, de^k\rrangle -  \langle \ad e_h, \tran\ad((e^k)^\sharp)^\natural\rangle\\
=
\frac14\llangle b^{jm}_h  e_j^\flat \wedge e_m^\flat, a_{il}^k e^{il}\rrangle
 -  \langle a_{hj}^i e^j\otimes e_i, b^{kj}_l e^l\otimes e_j\rangle =\frac12 a_{il}^k  b^{il}_h -a^i_{hj}b_i^{kj}.
\end{multline*}
The statement now follows from
\begin{align*}
a(e_i) \circ  b(e^i) &= e^h\otimes a(e_i) b^{ij}_h e_j = a_{ij}^k b^{ij}_h e^h\otimes  e_k,  \\
 b(e^i)\circ a( e_i) &= e^h\otimes b(e^i) a_{ih}^m  e_m = a_{ih}^m  b^{ik}_m e^h\otimes  e_k.\qedhere
 \end{align*}
\end{proof}
It will be convenient to make the dependence of the pair $(a,b)$ on the Lie algebra structure and  pseudoriemannian metric explicit. Let $\mathcal{S}\subset S^2T^*$ be the set of non-degenerate scalar products on $T$; there is a natural action of $\GL(T)$ on $\mathcal{S}$, with a finite number of orbits, one for each possible signature. Elements of $\mathcal{S}$ can be identified with isomorphisms
\[S\colon T\to T^*, \quad \tran S=S\]
by mapping a scalar product to its associated musical isomorphism.

We define the action of $g\in\GL(T)$ on $\mathcal{S}$ by the commutativity of the diagram
\[\xymatrix{
T\ar[d]^{g}\ar[r]^S & T^*  \ar[d]^{\tran{g}^{-1}}\\
T\ar[r]^{gS}& T^*
}\]
i.e.  $gS=\tran{g}^{-1}Sg^{-1}$. The correspondence~\eqref{eqn:musica} now reads
\[\tran b(Se_i) = S a(e_i)S^{-1},\]
namely $b=q(a,S)$, where
\[q\colon V^{11}_1\times \mathcal{S}\to V^1_{11}, \quad q(e^i\otimes a_i,S)=S^{-1}e^i\otimes S^{-1}\tran a_i S.\]
Notice that $q$ is linear in the first variable.
\begin{lemma}
\label{lemma:dq}
The map $q$ is equivariant and satisfies
\[dq_{a, S}(a',g'S)= q(a'-g'a,S)+g'q(a,S), \quad g'\in\gl(T).\]
\end{lemma}
\begin{proof}
Equivariance follows from
\begin{gather*}
 q(\tran{g}^{-1}e^i\otimes ga_ig^{-1},gS)=gS^{-1}\tran{g}\tran{g}^{-1}e^i\otimes gS^{-1}\tran{g}  \tran g^{-1} \tran{a_i}\tran{g}\tran{g}^{-1}Sg^{-1},\\
g q(e^i\otimes a_i,S)=gS^{-1}e^i\otimes gS^{-1}\tran a_i Sg^{-1}.
\end{gather*}
Using equivariance and linearity in the first variable, we compute
\[dq_{a, S}(a',g'S)=
dq_{a, S}(g'a,g'S)+dq_{a, S}(a'-g'a,0) = g'q(a,S)+q(a'-g'a,S).\qedhere\]
\end{proof}
We will also need the following observation:
\begin{lemma}
\label{lemma:symmetry}
Given $a',a''\in V^{11}_1$ and $S\in\mathcal{S}$, 
\[\langle a',q(a'',S)\rangle =\langle a'',q(a',S)\rangle .\]
\end{lemma}
\begin{proof}
From $\tran{S}=S$, we deduce
\begin{multline*}
\langle a',q(a'',S)\rangle =\langle e^j\otimes a'_j, S^{-1}(e^i)\otimes S^{-1}\tran a''_i S  \rangle\\
= e^j(S^{-1}(e^i)) \Tr a'_j S^{-1}\tran a''_i S
= e^j(S^{-1}(e^i)) \Tr S a''_i S^{-1}\tran a'_j\\
= e^j(S^{-1}(e^i)) \Tr  a''_i S^{-1}\tran a'_j S
= \langle e^i\otimes a''_i,S^{-1}(e^j)\otimes S^{-1}\tran a'_j S \rangle\\
=\langle a'',q(a',S)\rangle.\qedhere
\end{multline*}
\end{proof}

We can restate Proposition~\ref{prop:c1c2} as follows:
\begin{corollary}
\label{cor:riccic1c2}
Let $S\in\mathcal{S}$ define a metric on $T_a$, with $a\in\PP$; its Ricci operator satisfies
\[\ric=\frac14c_1(a,q(a,S))-\frac12c_2(a,q(a,S)).\]
\end{corollary}

On $V^{11}_1\times V^1_{11}$ we have a natural symplectic form given by
\[\omega_{a,b}((v,w),(v',w'))=\langle v,w'\rangle-\langle v',w\rangle,\]
where closedness of $\omega$ follows from the fact that is has constant coefficients. Up to a factor, the Ricci operator corresponds to an invariant bilinear map
\[\mu\colon V^{11}_1\times V^1_{11}\to\gl(T), \quad (a,b)\mapsto c_1(a,b)-2c_2(a,b).\]

\begin{proposition}
\label{prop:moment}
$\mu$ is the moment map for the action of $\GL(T)$ on $V^{11}_1\times V^1_{11}$ and it satisfies 
\[\langle \mu(a,b),X\rangle = \langle Xa,b\rangle, \quad a\in V^{11}_1, b\in V^1_{11}, X\in\gl(T).\]
\end{proposition}
\begin{proof}
We compute
\[d\mu_{a,b}(v,w)=c_1(a,w)-2c_2(a,w)+c_1(v,b)-2c_2(v,b).\]
The moment map condition reads
\[\langle d\mu_{a,b}(v,w),X\rangle = \omega_{a,b}((Xa,Xb),(v,w)),\]
or equivalently
\begin{align*}
\langle c_1(a,w), X\rangle-2\langle c_2(a,w), X\rangle &=\langle Xa,w\rangle, \\
\langle c_1(v,b), X\rangle-2\langle c_2(v,b),X\rangle &=-\langle v,Xb\rangle.
\end{align*}
Writing \[Xa=Xe^i \otimes a_i + e^i\otimes [X,a_i], \]
observing that $(Xe^i)(e_j)=-\langle X, e^i\otimes e_j\rangle$ and using \eqref{eqn:c2new}, we compute
\begin{align*}
\langle Xa,w\rangle&=-\langle c_2(a,w),X\rangle + \Tr ([X,a_i]w_i)
=-2\langle c_2(a,w),X\rangle + \langle c_1(a,w),X\rangle, \\
\langle v,Xb\rangle&=\langle c_2(v,b),X\rangle + \Tr (v_i[X,b_i])
=2\langle c_2(v,b),X\rangle -\langle c_1(v,b),X\rangle.
\end{align*}
From the definition of $\mu$ it follows that $\langle Xa,w\rangle$ equals $\langle \mu(a,w),X\rangle$.
\end{proof}
\begin{remark}
The action of $\GL(T)$ on $V^{11}_1\times V_{11}^1$ is not free. In fact, suppose $b=q(a,S)$; then $g\in\GL(T)$ fixes both $a$ and $b$ if and only 
\[[ge_i,ge_j]=g[e_i,e_j], \quad g[Se_i,Se_j]=[g(Se_i),g(Se_j)],\]
meaning that both $g$ and $S^{-1}\circ\tran{g}^{-1}\circ S$ must be automorphisms of $T_a$, i.e.
\[g\in \Aut(T_a)\cap (S^{-1}\tran(\Aut T_a)S).\]
It is easy to see that for fixed $T_a$, a suitable choice of $S$ makes this set not empty.
\end{remark}
\begin{remark}
For a fixed element $\lambda\id$ of the center of $\gl(n,\R)$, the space $\mu^{-1}(\lambda\id)/\GL(n,\R)$ has a natural interpretation as a symplectic reduction.  Einstein Lie algebras in $\PP$ form a subset of this reduction; it is plausible that this symplectic structure may give useful information on the moduli space of Einstein Lie algebras in $\PP$.

We emphasize that the group of symmetries $\GL(n,\R)$ is noncompact. There is a theory of moment maps for actions of noncompact Lie groups, see e.g. \cite{Heinzner:TheMinimumPrinciple}; in order to apply it, one would need to extend the action of the group $\GL(n,\R)$ to a holomorphic action of $\GL(n,\C)$ on a K\"ahler manifold. The only natural way of doing this is by complexifying $V^{11}_1\times V^1_{11}$; the K\"ahler form would have to be $\GL(n,\R)$-invariant element of
\[\Lambda^2((V^{11}_1+iV^{11}_1)\times (V^1_{11} + iV_{11}^1)),\]
of complex type $(1,1)$. However, such a K\"ahler form becomes trivial upon restricting to $V^{11}_1\otimes V^1_{11}$, whereas the theory requires that it restrict to $\omega$.

Nevertheless,  in order  to understand the symplectic structure of the space of Einstein Lie algebras, it might be useful to consider a paracomplexification rather than a complexification of $V^{11}_1\times V^1_{11}$. Indeed, this space has a natural paracomplex structure, namely a tensor of type $(1,1)$ given by
\[K(a,b)=(a,-b);\]
$K$ is compatible with the symplectic form $\omega$ and the pair $(K,\omega)$ induces a metric of split signature, which is parak\"ahler in the sense of \cite{HarveyLawson}. The $\GL(n,\R)$-action extends naturally to an action of 
\[\GL(n,\mathbf{D}) = \{M\in M^n(\mathbf{D})\mid \det M \in \mathbf{D}^*\},\]
where  $\mathbf{D}=\R[\tau]$, $\tau^2=1$ is the ring of double numbers. However, it seems difficult to carry over to this setting  the theory of \cite{Heinzner:TheMinimumPrinciple}, where positive definiteness of the metric is a key ingredient.
\end{remark}

The moment map $\mu$ allows us to generalize to the pseudoriemannian case the known interpretation of the Ricci operator in terms of group actions of \cite{Lauret}.  For $X\in\gl(T)$, let 
 ${X^+}$ denote the fundamental vector field  associated to the action 
\begin{equation}
\label{eqn:plusaction}
\GL(T)\times (V^{11}_1\times \mathcal{S})\to V^{11}_1\times \mathcal{S}, \quad g(a,S)=(ga,S).
\end{equation}
By Corollary~\ref{cor:riccic1c2} and \eqref{eqn:tracec1c2}, the scalar curvature is given by a functional
\[s\colon\PP\times \mathcal{S}\to \R, \quad s(a,S)=-\frac14 \langle a,q(a,S)\rangle.\]
  As an element of $\gl(T)$, the Ricci operator is determined by its contractions $\langle \ric, X\rangle$, where $X$ ranges in $\gl(T)$. It turns out that such a contraction corresponds to the Lie derivative of the scalar curvature along the direction defined by $X$:
\begin{theorem}
\label{thm:riccifromXplus}
For $a\in \mathcal{P}$,
the Ricci operator of a metric $S$ satisfies
\[\langle \ric_S,X\rangle =\frac14\langle Xa,q(a,S)\rangle, \quad X\in\gl(T);\]
Moreover,
\[\langle \ric_S,X\rangle  =-\frac12(X^+s)(a,S).\]
\end{theorem}
\begin{proof}
The first part follows from Proposition~\ref{prop:moment}.

For the second part, we compute
\[X^+s (a,S)=ds_{a,S}(Xa,0)=-\frac14\langle Xa,q(a,S)\rangle- \frac14\langle a,q(Xa,S)\rangle,\]
where we have used linearity of $q$ in the first variable (see also Lemma~\ref{lemma:dq}). The statement now follows from Lemma~\ref{lemma:symmetry}.
\end{proof}

This result will be used in next section (Theorem~\ref{thm:ostruzionegl}) to determine some conditions that must me satisfied by nilpotent Lie algebras that admit Einstein metrics of nonzero scalar curvature. For the moment, we use it to recover a result that was proved by Jensen for Riemannian metrics on unimodular Lie algebras.
\begin{corollary}[\cite{Jensen:scalar}]
\label{cor:jensen}
Given $a\in\PP$, the set of Einstein metrics on $T_a$ coincides with the set of critical points for the scalar curvature functional among metrics of constant volume.
\end{corollary}
\begin{proof}
By the theorem, a metric $S$ is Einstein if and only if 
\[0=(X^+s)(a,S), \quad X\in\Sl(T),\]
namely when $(a,S)$ is a critical point for $s$ in the $\SL(T)$-orbit, relative to the action \eqref{eqn:plusaction}. However, by equivariance,
\[s(ga,S)=s(a,g^{-1}S), \quad g\in\SL(T),\]
so the Einstein condition amounts to $S$ being a critical point of
\[S\mapsto s(a,S)\]
in its $\SL(T)$-orbit, namely in the space of metrics with the same signature and volume.
\end{proof}
This result puts constraints on continuous families of Einstein metrics on a fixed Lie algebra $T_a$, $a\in\PP$. Indeed, if the volume is fixed, all metrics in such a family have the same scalar curvature; moreover, if one metric in the family is Ricci-flat, then all of them are Ricci-flat.

Jensen's result can be generalized by considering variations of both Lie algebra and metric. Fix a volume form on $\Lambda^n T^*$, and denote by $\hat{\mathcal{S}}\subset\mathcal{S}$ the set of metrics with that volume. Accordingly, the scalar curvature functional $s$ restricts to
\[\hat s\colon V^{11}_1\times \hat{\mathcal{S}}\to\R.\]
\begin{corollary}
\label{cor:ghigi}
Let $(a,S)\in V^{11}_1\times\hat{\mathcal{S}}$ satisfy $\mu(a,q(a,S))=\lambda\id$. Then
\[d\hat s_{a,S}(a',w)=-\frac12\langle a',q(a,S)\rangle.\]
\end{corollary}
\begin{proof}
By Lemma~\ref{lemma:symmetry},
\[d\hat s_{a,S}(a',0)=-\frac14\langle a',q(a,S)\rangle -\frac14 \langle a,q(a',S)\rangle = -\frac12\langle a',q(a,S)\rangle.\]
Write the generic vector in $T_S\hat{\mathcal{S}}$ as $w=XS$, for $X\in\Sl(T)$; by equivariance,
\[d\hat s_{a,S}(0,XS)=d\hat s_{a,S}(-Xa,0)= \frac12\langle Xa,q(a,S)\rangle=\frac12\langle \mu(a,q(a,S)),X\rangle=0,\]
where we have used Proposition~\ref{prop:moment}.
\end{proof}
For $a'=0$, we recover our version of Jensen's result. If one further restricts $\hat s$ to $\mathcal{P}\times\hat{\mathcal{S}}$,  one can consider the set of critical points of $\hat s$; this is strictly contained in the set of Einstein Lie algebras, which are only critical points $(a,S)$ of the restriction to $\{a\}\times \hat{\mathcal{S}}$. Corollary~\ref{cor:ghigi} allows us to state this ``critical point'' condition as
\begin{equation}
 \label{eqn:ghigi}
\langle a',q(a,S)\rangle=0, \quad a'\in T_a\mathcal{P}.
\end{equation}
Notice however that such a condition is only satisfied by Ricci-flat metrics, because $a\in T_a\mathcal{P}$ and $\langle a, q(a,S)\rangle$ is a multiple of the scalar curvature.

\begin{example}
As shown in \cite{ContiRossi:ricci}, the metric $S=e^1\odot e^4+e^2\odot e^5+e^3\odot e^6$ on the nilpotent Lie group $T_a=(24,0,0,0,0,35)$  is Ricci-flat. We have
\[q(a,S)=e_1\otimes e^4\otimes e_5+e_2\otimes e^3\otimes e_6-e_5\otimes e^4\otimes e_1-e_6\otimes e^3\otimes e_2.\]
One verifies that given $a'$ tangent to the space of Lie algebras in $a$, one has $\langle a',\tilde q(a,S)\rangle=0$; thus, the critical point condition~\eqref{eqn:ghigi} is verified.
 
On the other hand, if we take $T_{a_\lambda}=(0,0,\lambda 12,0,0,45)$ with the metric $S=e^1\odot e^4+e^2\odot e^5+e^3\odot e^6$, we compute
\[c_1(a_\lambda,S) = 2\lambda(e^3\otimes e_3+e^6\otimes e_6), \quad  c_2(a_\lambda,S) = \lambda(e^1\otimes e_1+e^2\otimes e_2+e^4\otimes e_4+e^5\otimes e_5);\]
this is  Ricci-flat for $\lambda=0$;
$(a_0,S)$ does not satisfy ~\eqref{eqn:ghigi}, consistently with the fact that $\lambda=0$ is not a critical point for the scalar curvature $s(a_\lambda,S)$ as a function of $\lambda$.
\end{example}

We emphasize that the results of this section also apply to the Riemannian case. In subsequent sections we will consider the Einstein condition on nilpotent Lie groups, which only admit non-trivial examples in the indefinite case due to the results of Milnor \cite{Milnor:curvatures}.

\section{Further nonexistence results}
\label{sec:sixdim}
In this section we go back to the existence problem studied in Section~\ref{sec:fivedim}. Employing the methods of Section~\ref{sec:momentmap}, as well as the classification of $6$-dimensional nilpotent Lie algebras (see \cite{Magnin}), we prove that left-invariant Einstein metrics on nilpotent Lie groups of dimension $\leq6$ are Ricci-flat.

The key tool is a consequence of Theorem~\ref{thm:riccifromXplus}, that relates the existence of Einstein metrics on $\g$ to a property of the Lie algebra of derivations, defined as
\[\Der(\g)=\{X\colon \g\to \g\mid X \text{ is linear and } X[v,w]=[Xv,w]+[v,Xw]\ \forall v,w\in\g\}.\]
\begin{theorem}
\label{thm:ostruzionegl}
Let $\g$ be a unimodular Lie algebra with Killing form zero. If $\g$ has an Einstein metric with nonzero scalar curvature, then $\Der(\g)\subset\Sl(\g)$.
\end{theorem}
\begin{proof}
In the language of Section~\ref{sec:momentmap}, let $\lie{g}=T_a$; if $X\in\gl(n,\R)$ is a derivation, then $Xa=0$. For any metric $S$ on $T_a$, Theorem~\ref{thm:riccifromXplus} gives
\[\langle \ric,X\rangle =\frac14\langle Xa,q(a,S)\rangle=0;\]
if $S$ is Einstein, say $\ric=\lambda\id$, 
\[0=\langle \lambda\id, X\rangle = \lambda \Tr(X),\]
so either $\lambda=0$ or $\Tr(X)=0$.
\end{proof}

\begin{remark}
Every Lie algebra with $Z\not\subset \g^1$ is a direct sum of Lie algebras
\[\g=\lie{h}\oplus\R, \quad \lie{h}=\Span{e_1,\dotsc, e_{n-1}}, \R=\Span{e_n};\]
it is clear that $e^n\otimes e_n$ is a derivation with trace $1$; so Lemma~\ref{lemma:centroinderivato} can be viewed as a consequence of Theorem~\ref{thm:ostruzionegl}. Similary, a nilpotent Lie algebras of step two can be written as
\[\g=\Span{e_1,\dotsc, e_n}, \quad \g'=\Span{e_1,\dotsc, e_k}.\]
Then $\id+e^1\otimes e_1+\dots + e^k\otimes e_k$ is a derivation with trace $n+k$; thus, Corollary~\ref{cor:steptwo} also follows from Theorem~\ref{thm:ostruzionegl}.
\end{remark}

We can now prove:
\begin{theorem}
\label{thm:nil6}
On a nilpotent Lie algebra of dimension up to six, Einstein metrics are Ricci-flat.
\end{theorem}
\begin{proof}
Nilpotent Lie algebras of dimension $\leq6$ are classifed by Magnin~\cite{Magnin}. By Theorem~\ref{thm:ostruzionegl}, it suffices to verify that each Lie algebra has a derivation $X$ with nonzero trace. We illustrate this in the case of the 5-dimensional Lie algebra $(0,0,12,13,23)$, which has the form $T_a$ for
\[a=-e^{12}\otimes e_3 - e^{13}\otimes e_4-e^{23}\otimes e_5.\]
Setting
\[X=\operatorname{diag}(x_1,\dotsc, x_n)=x_1e^1\otimes e_1+\dots + x_ne^n\otimes e_n,\]
we find
\[Xa=(x_1+x_2-x_3)e^{12}\otimes e_3 +(x_1+x_3-x_4) e^{13}\otimes e_4+(-x_2-x_3+x_5)e^{23}\otimes e_5;\]
it is clear that $Xa=0$ has a solution with $\Tr(X)\neq0$, e.g.
\[X=e^1\otimes e_1 + e^3\otimes e_3+2e^4\otimes e_4+e^5\otimes e^5.\]
The same argument applies to all nilpotent Lie algebras of dimension up to six.
\end{proof}

In dimension 7, the same technique gives a weaker result:
\begin{theorem}
\label{thm:einstein_seven_in_table}
If $\g$ is a nilpotent $7$-dimensional Lie algebra not appearing in Table~\ref{table:nilpotenteinstein}, every Einstein metric on $\g$ is Ricci-flat.
\end{theorem}
\begin{proof}
Seven-dimensional nilpotent Lie algebras are classified by Gong \cite{Gong}; we refer to that classification as reproduced in \cite{ContiFernandez:calibrated}. Case by case calculations show that the Lie algebras in Gong's list that satisfy $\Der(\g)\subset\Sl(\g)$ are precisely those of Table~\ref{table:nilpotenteinstein}. 

We illustrate the computation for the 9 one-parameter families that appear in the classification, namely
\begin{alignat*}{3}
147E&\!=\!\left(0,0,0,e^{12},e^{23},- e^{13}, \lambda e^{26}-e^{15}- {(-1+\lambda)} e^{34}\right),&\hspace{-2em}\lambda\neq0,1;\\
1357M&\!=\!\left(0,0,e^{12},0,e^{24}+e^{13},e^{14},- {(-1+\lambda)} e^{34}+e^{15}+ \lambda e^{26} \right),&\lambda\neq0;\\
1357N&\!=\!\left(0,0,e^{12},0,e^{13}+e^{24},e^{14},e^{46}+e^{34}+e^{15}+ \lambda e^{23} \right);&\\
1357S&\!=\!\left(0,0,e^{12},0,e^{13},e^{24}+e^{23},e^{25}+e^{34}+e^{16}+e^{15}+ \lambda e^{26}\right),&\lambda\neq1;\\
12457N&\!=\left(0,0,e^{12},e^{13},e^{23},e^{24}+e^{15}, \lambda e^{25}+e^{26}+e^{34}-e^{35}+e^{16}+e^{14}\right);\hspace{-2em}&\\
123457I&\!=\left(0,0,e^{12},e^{13},e^{14}+e^{23},e^{15}+e^{24}, \lambda e^{25}- {(-1+\lambda)} e^{34}+e^{16}\right);\hspace{-2em}&\\
147E_1&\!=\left(0,0,0,e^{12},e^{23},- e^{13},2 e^{26}-2 e^{34}- \lambda e^{16} + \lambda e^{25}\right),&\lambda>1;\\
1357QRS_1&\!=\left(0,0,e^{12},0,e^{13}+e^{24},e^{14}-e^{23}, \lambda e^{26} +e^{15}-  {(-1+\lambda)}e^{34}\right),&\lambda\neq0;\\
12457N_2&\!=\left(0,0,e^{12},e^{13},e^{23},-e^{14}-e^{25},e^{15}-e^{35}+e^{16}+e^{24}+ \lambda e^{25} \right),&\lambda\geq0.
\end{alignat*}

For each of these families, consider the diagonal matrix $X\in\gl(n,\R)$,
\begin{equation*}
X=x_1e^1\otimes e_1+\dots + x_ne^n\otimes e_n;
\end{equation*}
then $X$ is a derivation of $147E$ when
\[x_1=x_7-x_5,x_2=x_7-x_6,x_3=-x_7+x_5+x_6,x_4=2 x_7-x_5-x_6;\]
in particular, $X$ can be chosen with nonzero trace. The same calculation for the other one-parameter families gives:
\begin{align*}
1357M&: & x_1&=\frac{1}{3}  x_6,x_2=x_7-x_6,x_3=x_7-\frac{2}{3} x_6,x_4=\frac{2}{3}  x_6,x_5=x_7-\frac{1}{3} x_6;\\
1357N&: & x_1&=\frac{1}{5}  x_7,x_2=\frac{2}{5}  x_7,x_3=\frac{3}{5}  x_7,x_4=\frac{2}{5}  x_7,x_5=\frac{4}{5}  x_7,x_6=\frac{3}{5}  x_7;\\
1357S&: & x_1&=\frac{1}{4}  x_7,x_2=\frac{1}{4}  x_7,x_3=\frac{1}{2}  x_7,x_4=\frac{1}{2}  x_7,x_5=\frac{3}{4}  x_7,x_6=\frac{3}{4}  x_7;\\
123457I&: & x_1&=\frac{1}{7}  x_7,x_2=\frac{2}{7}  x_7,x_3=\frac{3}{7}  x_7,x_4=\frac{4}{7}  x_7,x_5=\frac{5}{7}  x_7,x_6=\frac{6}{7}  x_7;\\
147E_1&: & x_1&=x_7-x_6,x_2=x_7-x_6,x_3=-x_7+2 x_6,x_4=2 x_7-2 x_6, x_5=x_6;\\
{\tiny1357QRS_1}&: & x_1&=\frac{1}{4}  x_7,x_2=\frac{1}{4}  x_7,x_3=\frac{1}{2}  x_7,x_4=\frac{1}{2}  x_7,x_5=\frac{3}{4}  x_7,x_6=\frac{3}{4}  x_7.
\end{align*}
This method fails for $12457N$ and $12457N_2$: in fact, all derivations of these Lie algebras are strictly lower triangular relative to the basis $e_1,\dotsc, e_8$, and therefore trace-free.

The statement follows from Theorem~\ref{thm:ostruzionegl}.
\end{proof}

\begin{table}[tp]
\centering
\begin{tabular}{RLC}
\multirow{3}{*}{\textbf{a}}& 123457E & (0,0,12,13,14,23+15,23+24+16)\\
& 123457H & (0,0,12,13,14+23,15+24,25+23+16)\\
& 123457H_1 & (0,0,12,13,14+23,15+24,-16-25+23)\\
\midrule
\multirow{7}{*}{\textbf{b}}& 13457I & ( 0,0,12,13,14,23,25+26-34+15)\\
& 12457J & ( 0,0,12,13,23,24+15,34+25+16+14)\\
& 12457J_1 & ( 0,0,12,13,23,24+15,34-25+16+14)\\
& 12457N & \left(0,0,{12},{13},{23},{24}+{15}, \lambda {25}+{26}+{34}-{35}+{16}+{14}\right)\\
& 12457N_1 & ( 0,0,12,13,23,-25-14,-35+25+16)\\
& 12457N_2 &\left(0,0,{12},{13},{23},-{14}-{25},{15}-{35}+{16}+{24}+ \lambda{25}\right), \lambda\geq0\\
\midrule
\textbf{c}& 123457F & (0,0,12,13,14,15+23,16-34+24+25)\\
\midrule
\textbf{d}& 12457G & ( 0,0,12,13,0,25+14+23, -34+26+15)
\end{tabular}
\caption{\label{table:nilpotenteinstein}Nilpotent Lie algebras of dimension 7 that might carry an Einstein metric with $s\neq0$. The letters \textbf{a},\textbf{b},\textbf{c},\textbf{d} refer to the families of Theorem~\ref{thm:nicenoeinstein}.}
\end{table}

We do not know whether the Lie algebras of Table~\ref{table:nilpotenteinstein} have an Einstein metric of nonzero scalar curvature; solving this problem amounts to solving a system of polynomial equations in $\binom{8}{2}$ variables for each Lie algebra; our attempts to attack this problem with software based on Gr\"obner bases methods have been unsuccesful. However, the number of parameters can be reduced considerably by considering metrics on which some orthonormal frame satisfies certain special conditions; this will be used in the last section of this paper to construct an Einstein metric in dimension~8.

\section{Nice nilpotent Lie algebras and the Einstein condition}
In this section we construct an explicit Einstein metric on a nilpotent Lie algebra with nonzero scalar curvature; our example has dimension 8 and belongs to the class of nice nilpotent Lie algebras. In dimension 7, we show that nice nilpotent Lie algebras do not admit Einstein metrics of nonzero scalar curvature.

Indeed, let $\{e_1,\dots,e_n\}$ be a basis for a nilpotent Lie algebra $\g$, with structural constants $a^k_{ij}$. Following \cite{LauretWill:Einstein}, we say that the basis $\{e_i\}$ is \emph{nice} if the following conditions hold:
\begin{itemize}
	\item for all $i < j$ there is at most one $k$ such that $a^k_{ij}\neq 0$;
	\item if $a^k_{ij}$ and $a^k_{lm}$ are nonzero then either $\{i, j\} = \{l,m\}$ or $\{i, j\}\cap\{l,m\}=\emptyset$.
\end{itemize}
The corresponding element $a\in V^{11}_1$ can be characterized in terms of the one-dimensional representations of the Cartan algebra of $\Sl(n,\R)$
\[\Gamma^i_j=\Span{e^i\otimes e_j}, \Gamma_{ij}=\Span{e_i\otimes e_j}, \Gamma_k=\Span{e_k}, \Gamma^k=\Span{e^k};\]
indeed, the nice condition  corresponds to the subspace
\[\bigoplus_{i,j} \Gamma^i_j\otimes\Gamma^{k_{ij}} \cap\bigoplus_{i,j} \Gamma^{ij}\otimes\Gamma_{h_{ij}}.\]
A nilpotent Lie algebra is said to be \emph{nice} if it admits a nice basis. This is not a very restrictive condition: as shown in  \cite{LauretWill:diagonalization}, every nilpotent Lie algebra of dimension $\leq6$ is nice except 
$(0,0,0,12,14,15+23+24)$.

The relevance of nice bases to the problem at hand stems from an observation of \cite{LauretWill:Einstein}, which is an immediate consequence of Proposition \ref{prop:ricci_in_peace}:
\begin{lemma}[\cite{LauretWill:Einstein}]
Let $\g$ be a nilpotent Lie algebra with a nice basis. For any pseudoriemannian metric on $\g$ for which the nice basis is orthogonal,  the Ricci tensor is diagonal with respect to that basis.
\end{lemma}

We are therefore able to obtain an example in dimension $8$, obtained as an extension of the nilpotent Lie algebra $137B$ of Gong's classification.
%
\begin{theorem}
\label{thm:einstein_example}
The nilpotent Lie algebra
\[(0,0,0,0,12+34,14-23,-24+35+16,-13+26+45)\] 
has four Einstein metrics with scalar curvature $\frac{56}{15}$, namely
\begin{gather}
e^1\!\otimes\! e^1\! +\! e^2\!\otimes\! e^2\!\pm\!(\!e^3\!\otimes\! e^3\!+\! e^4\!\otimes\! e^4\!)\!
-\!\tfrac73 e^5\!\otimes\! e^5\! \mp\!\tfrac73 e^6\!\otimes\! e^6\!
\pm\! \tfrac{98}{15}(\!e^7\!\otimes\! e^7 \!+ \! e^8\!\otimes\! e^8\! );\label{eqn:einsteindiagonal}\\
-e^1\odot e^2 \mp e^3\odot e^4 + \frac73 e^5\otimes e^5\pm \frac73e^6\otimes e^6 \pm\frac{98}{15}e^7\odot e^8.\label{eqn:einsteinsigma}
\end{gather}

These metrics are not locally symmetric and their holonomy is generic, i.e. respectively $\SO_+(6,2)$, $\SO_+(3,5)$,  $\SO_+(5,3)$ and $\SO_+(4,4)$.
\end{theorem}
\begin{proof}
Consider  the diagonal metric
\[g_1e^1\otimes e^1 + \dots + g_8 e^8\otimes e^8;\]
set $g_1=1$ for a normalization. Then
\begin{multline*}
 \ric = \frac12\operatorname{diag}\bigl(
-\frac{g_8}{g_3}-\frac{g_6}{g_4}-\frac{g_5}{g_2}-\frac{g_7}{g_6},
-\frac{g_7}{ g_2 g_4}-\frac{g_5}{g_2}-\frac{g_8}{ g_2 g_6}-\frac{g_6}{ g_2 g_3},\\
-\frac{g_7}{ g_5 g_3}-\frac{g_8}{g_3}-\frac{g_5}{ g_3 g_4}-\frac{g_6}{ g_2 g_3},
-\frac{g_8}{ g_5 g_4}-\frac{g_7}{ g_2 g_4}-\frac{g_6}{g_4}-\frac{g_5}{ g_3 g_4},\\
-\frac{g_7}{ g_5 g_3}-\frac{g_8}{ g_5 g_4}+\frac{g_5}{g_2}+\frac{g_5}{ g_3 g_4},
\frac{g_6}{g_4}-\frac{g_8}{ g_2 g_6}+\frac{g_6}{ g_2 g_3}-\frac{g_7}{g_6},\\
\frac{g_7}{ g_5 g_3}+\frac{g_7}{ g_2 g_4}+\frac{g_7}{g_6},
\frac{g_8}{ g_5 g_4}+\frac{g_8}{g_3}+\frac{g_8}{ g_2 g_6} \bigr).
\end{multline*}
Thus, if we take the metrics \eqref{eqn:einsteindiagonal}, we find $\ric=\frac{7}{15}\id$; \eqref{eqn:einsteinsigma} is obtained in a similar way.

Computations show that the curvature tensor is not parallel and its components span all of $\Lambda^2\g^*$, forcing the holonomy to be generic. 
\end{proof}
\begin{example}
The nilpotent Lie algebra
\begin{multline*}
\bigl(0,0,4 \sqrt{3}e^{12},-\sqrt{\frac{5}{2}}e^{13},\sqrt{\frac{5}{2}}e^{23},3 \sqrt{\frac{7}{2}}e^{24}-3 \sqrt{\frac{7}{2}}e^{15},\\
 \sqrt{21}e^{34}+2 \sqrt{3}e^{25}+4 \sqrt{2}e^{16},4 \sqrt{2}e^{26}+\sqrt{21}e^{35}+2 \sqrt{3}e^{14}\bigr)
\end{multline*}
admits the diagonal Lorentzian Einstein metric 
\[e^1\otimes e^1+\dots + e^5\otimes e^5 - e^6\otimes e^6+e^7\otimes e^7+e^8\otimes e^8;\]
in this case $s=4$.  The systematic construction of examples of this type will be illustrated in \cite{ContiRossi:DraftEinsteinNice}.
\end{example}

Together with Theorem~\ref{thm:einstein_example} and the fact that sign changes of the metric preserve the Einstein condition, this example gives the following:
\begin{corollary}
For any indefinite signature in dimension 8, there exists a nilpotent Lie algebra with an Einstein metric of nonzero scalar curvature of that signature.
\end{corollary}

It is natural to ask whether a similar construction can be applied to the $7$-dimensional case. The answer turns out to be negative,  as illustrated by the following:
\begin{theorem}
\label{thm:nicenoeinstein}
If $\g$ is a nilpotent $7$-dimensional Lie algebra admitting  a nice basis, every Einstein metric on $\g$ is Ricci-flat.
\end{theorem}
The proof uses the fact that the Lie algebras of Table~\ref{table:nilpotenteinstein} are not nice. To see this, we can divide the Lie algebras of Table~\ref{table:nilpotenteinstein} into four families characterized by the following conditions:
\begin{enumerate}[\bf a.]
	\item $(\dim \g^i) = (5,4,3,2,1)$, $\g^5 = Z$, $[\g^1,\g^1] = 0$;
	\item $(\dim \g^i)=(5,4,2,1)$,  $\g^4 = Z=[\g^1,\g^1] = [\g^1,\g^2] $,\mbox{ $[\g^2,\g^2]=0=[\g^1,\g^3]$};
	\item $(\dim \g^i)=(5,4,3,2,1)$, $\g^5 = Z=[\g^1,\g^1] = [\g^1,\g^2]$, $[\g^2,\g^2]=0=[\g^1,\g^3]$;
	\item $(\dim \g^i)=(4,3,2,1)$, $\g^4 = Z=[\g^1,\g^1] = [\g^1,\g^2]$, $[\g^2,\g^2]=0=[\g^1,\g^3]$.
\end{enumerate}
Here $(\dim \g^i)$ is short-hand notation for the dimensions of the lower central series $(\dim \g^1,\dots,\dim \g^k)$.

\begin{table}[tp]
\centering
\begin{tabular}{RLC}
\multirow{4}{*}{\textbf{a}} & 123457A & (0,0, 12, 13, 14, 15, 16) \\
& 123457B & (0,0, 12, 13, 14, 15, 16 +23)\\
& 123457D & (0,0, 12, 13, 14, 15 +23, 16 +24)\\
& 123457I (\lambda=1) & (0,0,12,13,14 + 23, 15 + 24, 16 + 25)\\
\midrule
\multirow{4}{*}{\textbf{b}} 
& 12457H & (0,0, 12, 13, 23, 15 + 24, 34 + 16)\\
& 12457I & (0,0, 12, 13, 23, 15 + 24, 34 + 16 +25)\\
& 12457L & (0,0, 12, 13, 23 , 14 - 25, 34 - 26) \\
& 12457L_1 & (0,0, 12, 13, 23 , 14 + 25, 34 - 26) \\
\midrule
\multirow{2}{*}{\textbf{c}} & 123457C & (0,0,12,13, 14 , 15, 16+34-25) \\
&123457I (\lambda\neq1) & (0,0,12,13, 14 +23, 15 +24, 16+(1-\lambda) 34+\lambda  25) \\
\midrule
\multirow{4}{*}{\textbf{d}} &12457C & (0,0,0, 12, 24, 13 + 25, 16 +45)\\
& 12457D & (0,0,0, 12, 24, 13 + 25,23+ 16 +45)\\
&13457C & (0,0,0, 12, 24, 25, 23 + 16 +45)\\
& 13457E & (0,0,0, 12, 24, 14 + 25, 23+ 16 +45)
\end{tabular}
\caption{\label{table:niceinfamilies}Nice nilpotent Lie algebras in the families \textbf{a}, \textbf{b}, \textbf{c} and \textbf{d}. The second column refers to Gong's classification; the correspondence may involve a change of basis.} 
\end{table}
\begin{lemma}
\label{lemma:niceinfourfamilies}
The $7$-dimensional nice nilpotent Lie algebras that satisfy the conditions~\textbf{a}, \textbf{b}, \textbf{c} or \textbf{d} are those of Table~\ref{table:niceinfamilies}.
\end{lemma}
\begin{proof}
We construct explicitly the Lie algebras in each family. For each condition from~\textbf{a} to~\textbf{d}, suppose there exists a nice basis $\{e_1, \dots, e_7\}$, that we can assume to be compatible with the lower central series in the sense that each $\g^k$ is spanned by the last $\dim \g^k$ elements of the basis; this implies that each $[e_i, e_j]$ is a multiple of some $e_h$ with $h>i,j$.

In the first three cases $e_3$ belongs to $\g^1$; since multiplying $e_3$ by a nonzero constant does not affect the nice basis condition, we can assume that $[e_2,e_1]=e_3$. 

\textbf{Case a.} 
As $e_4\in\g_2$, we can assume one of $[e_3,e_1]$, $[e_3,e_2]$ equals $e_4$; up to interchanging $e_1$ and $e_2$, we can assume $[e_3,e_1]=e_4$. Set $Y_3=[e_3,e_2]$, $Y_4=[e_4,e_2]$; the Jacobi identity for $\{e_1,e_2,e_3\}$ yields
\begin{equation}
 \label{eqn:jacobi123}
0=[[e_1,e_2],e_3]+[[e_2,e_3],e_1]+[[e_3,e_1],e_2]=Y_4-[Y_3,e_1].
\end{equation}
Because the basis is nice and $[e_3,e_1]=e_4$, $Y_3$ is in $\g^3$; by \eqref{eqn:jacobi123}, $Y_4=[Y_3,e_1]\in\g^4$. Since $e_5\in\g^3$, we can assume $[e_4,e_1]=e_5$.

If $Y_4=0$, the Jacobi identity on $\{e_1,e_2,e_4\}$ gives $[e_5,e_2]=0$. By the same token  we have that $[e_5,e_1]=e_6$, and  the Jacobi identity on $\{e_1,e_2,e_5\}$ implies $[e_6,e_2]=0$, so that $[e_6,e_1]=e_7$. The nice basis condition implies that $Y_3\in\g^2$ is a multiple of $e_5$, $e_6$, or $e_7$, but \eqref{eqn:jacobi123} rules out the first two. We therefore obtain the Lie algebra \[(0,0,12,13,14,15,16+\lambda 23), \quad \lambda\in\R.\]
If $\lambda\neq0$, we can change to the basis $\{e^1, \lambda e^2,\lambda e^3,\lambda e^4,\lambda e^5, \lambda e^6,\lambda e^7\}$ and reduce to the case $\lambda=1$; this produces the first two entries of Table~\ref{table:niceinfamilies}.

If $Y_4$ spans $Z$, say $Y_4=[e_4,e_2]=e_7$, then $[e_6,e_2]=0$ and $[e_6,e_1]=\gamma e_7$, $\gamma\neq0$. The Jacobi identity on $\{e_1,e_2,e_4\}$ gives $[e_5,e_2]=0$, so we can assume $[e_5,e_1]=e_6$. Equation \eqref{eqn:jacobi123} gives $[e_3,e_2]=\frac{1}{\gamma}e_6$, and the resulting Lie algebra is 
\[(0,0,12,13,14,15+\frac{1}{\gamma}23,\gamma 16+24);\]
a suitable change of basis gives $123457D$.

Suppose now that $Y_4$ is not in $Z$, say  $Y_4=[e_4,e_2]=e_6$. Again, $Y_3$ is a multiple of $e_5$, $e_6$, or $e_7$, but the nice basis condition and \eqref{eqn:jacobi123} rule out the last two; so $Y_3=\mu e_5$, $\mu\neq0$ and $[\mu e_5,e_1]=e_6$; consequently, $[e_2,e_5]$ is proportional to $e_7$. The Jacobi identity on $\{e_1,e_2,e_5\}$ implies $[e_6,e_2]=0$; thus, we can write $[e_6,e_1]= e_7$. Finally, apply the Jacobi identity on $\{e_1,e_2,e_4\}$, which gives $[e_5,e_2]= e_7$; we find
\[(0,0,12,13,14+\mu 23, \frac{1}{\mu}15+24, 16 + 25).\]
The parameter can be eliminated with a change of basis, giving $123457I(\lambda=1)$.

\textbf{Case b.} By hypothesis, $e_7$ lies in $[e_3,\g^2]$; with no loss of generality, suppose that $[e_4,e_3]=e_7$ and every other Lie bracket in $\g^1$ is zero. The basis elements $e_4$ and $e_5$ can only be obtained as $[e_3,e_1]$, $[e_3,e_2]$; interchanging $e_1$ and $e_2$ if needed, we can assume  $[e_3,e_1]=e_4$, $[e_3,e_2]=e_5$.  Writing down the Jacobi identity on $\{e_1,e_2,e_3\}$, we obtain that $[e_5,e_1]$ is equal to $[e_4,e_2]$, and therefore a multiple of $e_6$, as $[e_4,e_3]=e_7$.

Suppose that $[e_5,e_1]=e_6=[e_4,e_2]$; then, $[e_4,e_1]$ is forced to be zero, and  the Jacobi identity on $\{e_1,e_2,e_4\}$ gives $[e_6,e_1]=e_7$, hence $[e_6,e_2]=0$. Finally, the nice basis condition implies that $[e_5,e_2]=\gamma e_7$, $\gamma\in\R$, leading to 
\[(0,0, 12, 13, 23, 15 + 24, 34 + 16 + \gamma 25);\]
this is $12457H$ for $\gamma=0$, and otherwise isomorphic to $12457I$.

Suppose now that $[e_5,e_1]=0=[e_4,e_2]$. The nice basis condition implies that $[e_4,e_1]$ is a multiple of $e_6$; the Jacobi identity on $\{e_1,e_2,e_4\}$ gives $[[e_4,e_1],e_2]=-e_7$, so we can assume $[e_4,e_1]=e_6$, $[e_6,e_2]=-e_7$, implying in turn that $[e_6,e_1]=0$ and $[e_5,e_2]=\lambda e_6$. Thus, the Lie algebra has structure equations 
\[(0,0, 12, 13, 23 , 14 + \lambda 25, 34 - 26),\]
where $\lambda\neq0$ because $e_5$ is not in the centre. It is now a matter of changing the basis to obtain $12457L$ when $\lambda>0$, and $12457L_1$ otherwise.

\textbf{Case c.} Arguing as in case \textbf{a},  we can assume $[e_2,e_1]=e_3$ and $[e_3,e_1]=e_4$; moreover the conditions $[\g^1,\g^2]=Z$ and $[\g^1,\g^3]=0$ give $[e_4,e_3]=\gamma e_7$, $\gamma\neq0$. Because the basis is compatible with the lower central series, $e_5$ is a multiple of either $[e_4,e_1]$ or $[e_4,e_2]$, and $e_6$  of  $[e_5,e_1]$ or $[e_5,e_2]$.

Suppose that $[e_5,e_2]=e_6$; then $[e_5,e_1]$ is a multiple of $e_7$ and the Jacobi identity on $\{e_1,e_2,e_5\}$ gives $[e_6,e_1]=0$, so we can assume that $[e_6,e_2]=e_7$. By the Jacobi identity on $\{e_1,e_2,e_4\}$, $e_5$ cannot be a multiple of $[e_4,e_1]$, so it must be a multiple of $[e_4,e_2]$; but then $[e_2,e_3]=0$, contradicting the Jacobi identity on $\{e_1,e_2,e_3\}$.

This shows that $[e_5,e_1]=e_6$, and the same Jacobi identities as before imply that, up to rescalings, $[e_6,e_2]=0$, $[e_6,e_1]=e_7$ and  $[e_4,e_1]=e_5$, so that $[e_4,e_2]$ is a multiple of $e_6$. The Jacobi identity on $\{e_1,e_2,e_3\}$ implies that $[e_2,e_3]$ is not a multiple of $e_6$. Summing up, the nice basis condition gives 
\[[e_3,e_2]=\alpha e_5,\quad [e_4,e_2]=\beta e_6, \quad [e_5,e_2]=\nu e_7, \quad \alpha, \beta, \nu\in\R;\] 
the Jacobi identity is only satisfied if $\alpha=\beta=\gamma+\nu$, resulting in the Lie algebra
\[(0,0,12,13, 14 +\alpha 23, 15 +\alpha24, 16+\gamma34+(\alpha-\gamma)25).\]
This is isomorphic to $123457C$ if $\alpha=0$, and otherwise to $123457I$, with $\lambda=1-\frac{\gamma}{\alpha}$.

\textbf{Case d.} The characterization of this case and the nice basis condition imply that the only non-vanishing Lie bracket in $\{e_4,\dots e_7\}$ has the form $[e_5,e_4]=e_7$. We can write $[e_2,e_1]=e_4$, because $e_4\in\g^1$. Furthermore, since both $[e_5,e_2]$ and $[e_5,e_1]$ are multiples of $e_6$, we can assume $[e_5,e_1]=0$; the Jacobi identity on $\{e_1,e_2,e_5\}$ then gives $e_7=[[e_5,e_2],e_1]$, and it is no loss of generality to assume $[e_6,e_1]=e_7$, $[e_5,e_2]=e_6$. 

Since $e_5\in\g^2$, we have $[e_4,Y]=e_5$, where $Y$ is a multiple of either $e_1$, $e_2$ or $e_3$. The Jacobi identity on $\{e_1,e_2,e_4\}$ gives $0=[[e_4,e_2],e_1]=[[e_4,e_1],e_2]$,  ruling out the first case,  and the third is ruled out by the Jacobi identity on $\{e_2,e_3,e_4\}$, because $[[e_4,e_3],e_2] = [[e_2,e_3],e_4]+[[e_4,e_2],e_3]\in Z$. The only possibility is then $[e_4,e_2]=\mu e_5$. 

The nice basis condition implies $[e_3,e_2]=\gamma e_7$, $[e_4,e_3]=\lambda e_6$; the Jacobi identity for  $\{e_1,e_2,e_3\}$ and $\{e_1,e_3,e_4\}$ imply respectively $[[e_3,e_1],e_2]=\lambda e_6$ and $[[e_1,e_3],e_4]=\lambda e_7$. Thus, $[e_4,e_3]=0$ and $[e_3,e_1]$ has no component along $e_5$. Again, by the nice basis condition we can write $[e_3,e_1]=\alpha e_6$, $[e_4,e_1]=\beta e_6$, where one of $\alpha$ and $\beta$ must be zero; moreover, since $e_3\notin Z$, $\alpha$ and $\gamma$ cannot both be zero. The Lie algebra has structure equations 
\[(0,0,0, 12,\mu 24, \alpha 13 + \beta 14 + 25, \gamma 23 +16 +45).\]
For each of the four possible cases: 
\[\alpha\neq0,\beta=\gamma=0; \quad \alpha,\gamma\neq0,\beta=0;\quad  \beta,\gamma\neq0,\alpha=0; \quad \gamma\neq0, \alpha=\beta=0,\] 
we can change the basis to obtain one of the Lie algebras listed in Table~\ref{table:niceinfamilies}.

It is now easy to verify the correspondence between the Lie algebras we have constructed and the list of Gong \cite{Gong}.
\end{proof}

\begin{proof}[Proof of Theorem~\ref{thm:nicenoeinstein}]
By Theorem~\ref{thm:einstein_seven_in_table}, it suffices to show that the Lie algebras of Table~\ref{table:nilpotenteinstein} are not nice. By Lemma~\ref{lemma:niceinfourfamilies}, any nice Lie algebra in  Table~\ref{table:nilpotenteinstein} should also belong to  Table~\ref{table:niceinfamilies}; since Lie algebras with different labels in Gong's classification are not isomorphic over $\R$, the intersection is empty.
\end{proof}

\smallskip
\textbf{Acknowledgments}. We thank G.~Calvaruso, A.~Fino, A.C.~Ghigi, H.~Kasuya, J.~Lauret and the referee for useful remarks.

\bibliography{einstein}

\end{document}